\documentclass[12pt,oneside,english]{amsart}
\usepackage[T1]{fontenc}
\usepackage[latin9]{inputenc}
\usepackage{geometry}
\usepackage{amsthm}
\usepackage{amstext}
\usepackage{amssymb}

\makeatletter
\numberwithin{equation}{section}
\numberwithin{figure}{section}
\theoremstyle{plain}
\newtheorem{thm}{\protect\theoremname}
  \theoremstyle{definition}
  \newtheorem{problem}[thm]{\protect\problemname}
  \theoremstyle{plain}
  \newtheorem{lem}[thm]{\protect\lemmaname}
  \theoremstyle{remark}
  \newtheorem*{rem*}{\protect\remarkname}

\makeatother

\usepackage{babel}
  \providecommand{\lemmaname}{Lemma}
  \providecommand{\problemname}{Problem}
  \providecommand{\remarkname}{Remark}
\providecommand{\theoremname}{Theorem}

\begin{document}

\title{Bounds For The Tail Distribution of The Sum of Digits of Prime Numbers}

\author{Eric Naslund}
\begin{abstract}
Let $s_{q}(n)$ denote the base $q$ sum of digits function, which
for $n\leq x$, is centered around $\frac{q-1}{2}\log_{q}x$. In \cite{DrmotaMauduitRivat},
they look at sum of digits of prime numbers, and provide asymptotics
for the size of the set $\left\{ p\leq x,\ p\begin{tabular}{c}
 prime\end{tabular}s_{q}(p)=\alpha\left(q-1\right)\log_{q}x\right\} $ where $\alpha$ lies in the range 
\[
\alpha\in\left(\frac{1}{2}-K\frac{\left(\log\log x\right)^{\frac{1}{2}-\epsilon}}{\sqrt{\log x}},\ \frac{1}{2}+K\frac{\left(\log\log x\right)^{\frac{1}{2}-\epsilon}}{\sqrt{\log x}}\right)
\]
 for some constant $K$. In this note, we examine the tails of this
distribution, and prove that 
\[
\left|\left\{ p\leq x,\ p\begin{tabular}{c}
 prime\end{tabular}:\ s_{q}(p)\geq\alpha\left(q-1\right)\log_{q}x\right\} \right|\gg_{\epsilon}\ x^{2\left(1-\alpha\right)}e^{-c\left(\log x\right)^{1/2+\epsilon}}
\]
for $\frac{1}{2}\leq\alpha<0.7375$. This proves that there are infinitely
many primes with more than twice as many ones than zeros in their
binary expansion. 
\end{abstract}
\maketitle

\section{Introduction}

A prime number which can be written in the form $2^{n}-1$ will have
only ones in its binary expansion, and is called a Mersenne prime.
The first few such primes are $3$, $7$, $31$, and $127$. Currently,
the largest known prime is of this form, and it has over $12.9$ million
digits. These numbers have been looked at for centuries, and date
back to Euclid who was interested in them for their connection with
perfect numbers, something that we will not explore here. It is a
long standing conjecture that there are infinitely many Mersenne primes,
and currently this seems entirely out of reach of modern analytic
methods. However, we may weaken the condition and ask about primes
with a large number of $1$'s in their base $2$ expansion . With
this in mind, we ask the following motivational question:
\begin{problem}
\label{Main Question}Are there infinitely many primes with twice
as many ones than zeros in their binary expansion?
\end{problem}
If we let $s_{q}(n)$ denote the sum of the digits of $n$ written
in base $q$, then we are asking if there are infinitely many primes
$p$ which satisfy $s_{2}(p)\geq\frac{2}{3}\log_{2}p$. Moving to
a slightly more general setting, we will look at the sum of digits
base $q$ rather than just the binary case. The average of $s_{q}(n)$
is roughly $\frac{q-1}{2}$ multiplied by the number of digits, so
we have the asymptotic 
\[
\sum_{n\leq x}s_{q}(n)\sim\frac{q-1}{2}\log_{q}x.
\]
However, things become much more complicated when we restrict ourselves
to the prime numbers. In 1946 Copeland and Erdos \cite{ErdosCopeland}
proved that 
\[
\frac{1}{\pi(x)}\sum_{p\leq x}s_{q}(p)\sim\frac{q-1}{2}\log_{q}(x)
\]
where $\pi(x)=\sum_{p\leq x}1$ is the prime counting function, and
a more precise error term was subsequently given by Shiokawa \cite{Shiokawa}.
In 2009, Drmota, Mauduit and Rivat \cite{DrmotaMauduitRivat} gave
exact asymptotics for the set 
\[
\left\{ p\leq x,\ p\begin{tabular}{c}
 prime\end{tabular}s_{q}(p)=\alpha\left(q-1\right)\log_{q}x\right\} 
\]
 where $\alpha$ lies in the range 
\[
\alpha\in\left(\frac{1}{2}-K\frac{\left(\log\log x\right)^{\frac{1}{2}-\epsilon}}{\sqrt{\log x}},\ \frac{1}{2}+K\frac{\left(\log\log x\right)^{\frac{1}{2}-\epsilon}}{\sqrt{\log x}}\right),
\]
and is chosen so that $\alpha\left(q-1\right)\log_{q}x$ is an integer
which avoids certain congruence conditions. However, these results
don't allow us to make any conclusions about problem \ref{Main Question}.
In \cite{DrmotaMauduitRivat} they also asked about finding non-trivial
bounds for the sum $\sum_{p\leq x}2^{s_{q}(p)}$, as this would yields
results regarding the tail distribution of the sum of digits of primes.
That is, lower bounds for the size of sets of primes of the form 
\[
\left\{ p\leq x,\ p\ \text{prime}:\ s_{q}(n)\geq\alpha(q-1)\log_{q}x\right\} 
\]
where $\alpha>\frac{1}{2}$. These are exactly the type of bounds
we are looking for in order to answer our question, as problem \ref{Main Question}
is the case when $\alpha=\frac{2}{3}$ and $q=2$ . In this note,
we provide such lower bounds, and prove the following:
\begin{thm}
\label{thm: Main Theorem}Given $0.2625<\beta\leq\frac{1}{2}$ and
$\frac{1}{2}\leq\alpha<0.7375$, for sufficiently large $x$ we have
that 
\[
\left|\left\{ p\leq x,\ p\ \text{prime}:\ s_{q}(n)\geq\alpha(q-1)\log_{q}x\right\} \right|\gg_{\epsilon}\ x^{2\left(1-\alpha\right)}e^{-c\left(\log x\right)^{1/2+\epsilon}}
\]
and
\[
\left|\left\{ p\leq x,\ p\ \text{prime}:\ s_{q}(n)\leq\beta(q-1)\log_{q}x\right\} \right|\gg_{\epsilon}\ x^{2\beta}e^{-c\left(\log x\right)^{1/2+\epsilon}}.
\]

\end{thm}
We do not examine the sum $\sum_{p\leq x}2^{s_{q}(p)}$, rather we
note that the multinomial distribution is sharply peaked, so results
regarding primes in small intervals allow us to attain such a lower
bound. From theorem \ref{thm: Main Theorem}, problem \ref{Main Question}
follows as a corollary. In fact, we have that for any $\alpha<0.7375$
there are infinitely many primes where the proportion of $1$'s in
their binary expansion greater than $\alpha$.

\section{The Tail Distribution}

We start by providing bounds on the size of the tails of the multinomial
distribution. 
\begin{lem}
\label{lem:Chernoff-bound}(Chernoff bound) Given $\frac{1}{2}<a<1$,
we have that 
\[
\left|\left\{ n\leq q^{k}:\ a\left(q-1\right)k\leq s_{q}(n)\right\} \right|\leq\exp\left(-\frac{k}{18}\left(a-\frac{1}{2}\right)^{2}\right).
\]
\end{lem}
\begin{proof}
On the interval $\left[0,q^{k}\right]$ each digit can be thought
of as an independent random variable which corresponds to the roll
of a $q$ sided dice with sides $0,1,\dots,q-1$. Normalizing, let
$\xi$ be a random variable where 
\[
\text{P}\left(\xi_{i}=\frac{2}{q-1}j-1\right)=\frac{1}{q}
\]
for $0\leq j\leq q-1$, and for each $i$ let $\xi_{i}=\xi$. Our
goal is then to examine 
\[
\text{P}\left(\gamma\leq\frac{\xi_{1}+\xi_{2}+\cdots+\xi_{k}}{k}\right).
\]
For any nonnegative $t$, 
\begin{eqnarray*}
\text{P}\left(\gamma\leq\frac{\xi_{1}+\xi_{2}+\cdots+\xi_{k}}{k}\right) & \leq & \frac{\mathbb{E}\left(e^{t\left(\xi_{1}+\cdots+\xi_{k}\right)}\right)}{e^{tk\gamma}}\\
 & = & \left(e^{-t\gamma}\mathbb{E}\left(e^{t\xi}\right)\right)^{k}\\
 & = & e^{-kI\left(t,\gamma\right)}
\end{eqnarray*}
where 
\[
I\left(t,\gamma\right)=t\gamma-\log\mathbb{E}\left(e^{t\xi}\right).
\]
Evaluating the expectation, we find that 
\[
\mathbb{E}\left(e^{t\xi}\right)=\sum_{j=0}^{q-1}\frac{1}{q}e^{t\left(\frac{2j}{q-1}-1\right)}=\frac{e^{-t}}{q}\sum_{j=0}^{q-1}\left(e^{\frac{2t}{q-1}}\right)^{j}=\frac{1}{q}\frac{\sinh\left(t+\frac{t}{q-1}\right)}{\sinh\left(\frac{t}{q-1}\right)}.
\]
This gives rise to the series expansion
\[
\log\left(\frac{1}{q}\frac{\sinh\left(t+\frac{t}{q-1}\right)}{\sinh\left(\frac{t}{q-1}\right)}\right)=\frac{(q+1)}{6(q-1)}t^{2}-\frac{q^{3}+q^{2}+q+1}{180(q-1)^{3}}t^{4}+O\left(t^{6}\right),
\]
allowing us to prove that

\[
\log\mathbb{E}\left(e^{t\xi}\right)\leq\frac{(q+1)}{6(q-1)}t^{2}.
\]
To maximize $I(t,\gamma)$, we choose $t=\frac{\gamma}{3}\frac{q-1}{q+1},$
and obtain the upper bound 
\[
\text{P}\left(\gamma\leq\frac{\xi_{1}+\xi_{2}+\cdots+\xi_{k}}{k}\right)\leq\exp\left(-\frac{k}{6}\left(\frac{q-1}{q+1}\right)\gamma^{2}\right),
\]
which proves the lemma since $q\geq2$.
\end{proof}
Next, we will need the best existing results on prime gaps. In 2001,
Baker, Harman and Pintz proved that 
\begin{equation}
\pi\left(x+x^{\theta}\right)-\pi(x)\gg\frac{x^{\theta}}{\log x}\label{eq:Baker Harman Pintz}
\end{equation}
for any $\theta\geq0.525$ \cite{BakerHarmanPintz}. Armed with equation
\ref{eq:Baker Harman Pintz} and lemma \ref{lem:Chernoff-bound},
we are now ready to prove theorem \ref{thm: Main Theorem}.
\begin{proof}
Let $\alpha^{'}=\alpha+r(x)$ where $r(x)$ is chosen so that $\alpha^{'}<0.7375$.
Let $k=\left[\log_{q}x\right]$, so that $q^{k}\leq x$, and let $l=\lceil2\left(1-\alpha^{'}\right)k\rceil$.
Consider the interval $\left[q^{k}-q^{l},\ q^{k}-1\right]$, which
is an interval whose first $k-l$ digits base $q$ are equal to $q-1$.
By Baker, Harman and Pintz, there will be 
\[
\gg\frac{q^{l}}{\log\left(q^{k}\right)}\gg\frac{q^{l}}{\log x}
\]
primes in this interval, where the constant is explicit.. By Lemma
\ref{lem:Chernoff-bound}, there are at most $\exp\left(-\frac{l\delta^{2}}{18}\right)$
integers between $0$ and $q^{l}$ which have digit sum less than
$(q-1)l\left(\frac{1}{2}-\delta\right)$. Letting $\delta=\frac{\log l}{\sqrt{l}}$,
it follows that there are at most $q^{l}e^{-\left(\log l\right)^{2}}$
integers in the interval $\left[q^{k}-q^{l},\ q^{k}-1\right]$ whose
digit sum is less than 
\[
(q-1)(k-l)+(q-1)l\left(\frac{1}{2}-\frac{\log l}{\sqrt{l}}\right).
\]
As $q^{l}e^{-\left(\log l\right)^{2}}$ is significantly smaller than
$\frac{q^{l}}{k\log q}$, almost all of the primes in this interval
will have a digit sum greater than the above, and so we see that there
\[
\gg\frac{q^{l}}{\log\left(x\right)}
\]
primes with digit sum larger than 
\[
\alpha^{'}(q-1)k\log_{q}(x)-(q-1)\sqrt{l}\log l.
\]
Expanding $\alpha^{'}=\alpha+r(x)$, and taking $r(x)=c\frac{\log\log x}{\sqrt{\log x}}$
for the appropriate constant $c$ yields a digit sum greater than
\[
\alpha(q-1)\log_{q}(x),
\]
which proves the result since 
\[
\frac{q^{l}}{\log\left(x\right)}\sim\frac{x^{2(1-\alpha)}x^{-2r(x)}}{\log x}\gg x^{2(1-\alpha)}\exp\left(-c\sqrt{\log x}\log\log x\right).
\]
The proof for the lower bound of the size of the corresponding set
of primes with $s_{q}(p)\leq\beta(q-1)\log_{q}(x)$ for $0.2625<\beta\leq\frac{1}{2}$
is identical\@. \end{proof}
\begin{rem*}
The reader may note that for any $\alpha<0.7375$ there are more possible
choices for the first $k-l$ digits other than all $1$'s. It is conceivable
that if we looked at multiple intervals where the first $k-l$ digits
had many $1$'s that we would be able to increase the density by a
small factor, and possibly a significant factor for smaller $\alpha$.
While such an approach seems promising, and while it seems logical
to sum over multiple intervals, the end result and lower bound for
the number of primes is roughly the same. The exponent of $x$ is
no different, so we opted to present the simpler argument above. 
\end{rem*}

\specialsection*{Acknowledgements}

I would like to thank Didier Piau for helping me understand the Chernoff
bound.

\bibliographystyle{plain}

\end{document}